\documentclass[12pt]{article}
\usepackage{amsmath, amsfonts, amssymb}
\usepackage{graphicx, amsthm}
\usepackage{cite}
\usepackage{color}
\usepackage{enumerate}
\providecommand{\U}[1]{\protect\rule{.1in}{.1in}}

\newtheorem{theorem}{Theorem}[section]
\newtheorem{lemma}[theorem]{Lemma}

\theoremstyle{definition}
\newtheorem{definition}[theorem]{Definition}

\theoremstyle{remark}
\newtheorem{remark}[theorem]{Remark}
\numberwithin{equation}{section}

\usepackage[margin=1.3in, top=1.3in, bottom=1.3in]{geometry}

\def\R{{\bf R}}

\def\S{{\bf S}}
\def\E{\mathbb{E}}
\def\gph{\textrm{gph}\,}

\def\tr{\textrm{tr}\,}

\def\cl{\textrm{cl}\,}
\def\Diag{\textrm{Diag}\,}
\def\diag{\textrm{diag}\,}
\def\conv{\textrm{conv}\,}

\newcommand{\argmin}{\operatornamewithlimits{argmin}}
\newcommand{\ip}[1] {\left \langle #1 \right \rangle }

\def\E{{\bf E}}

\def\B{\mathcal{B}}

\renewcommand\footnotemark{}
\title{Variational analysis of spectral functions simplified \thanks{University of Washington, Department of Mathematics, 
		Seattle, WA 98195; Research of Drusvyatskiy and Kempton was partially supported by the AFOSR YIP award FA9550-15-1-0237.}}

\begin{document}

\author{D. Drusvyatskiy *
	\thanks{*
		E-mail: ddrusv@uw.edu;	\texttt{http://www.math.washington.edu/{\raise.17ex\hbox{$\scriptstyle\sim$}}ddrusv/}}
\and
$\textrm{C. Kempton }^{\dagger} $
\thanks{$\dagger$
	E-mail: yumiko88@uw.edu;			
}}



\date{\vspace{-5ex}}
\maketitle

\bigskip
\noindent\textbf{Abstract.}  
Spectral functions of symmetric matrices -- those depending on matrices only through their eigenvalues -- appear often in optimization.  
 A cornerstone variational analytic tool for studying such functions is a formula relating their subdifferentials to the subdifferentials of their diagonal restrictions. 
This paper presents a new, short, and revealing derivation of this result. We then round off the paper with an illuminating derivation of the second derivative of $C^2$-smooth spectral functions, highlighting the underlying geometry. All of our arguments have direct analogues for spectral functions of Hermitian matrices, and for singular value functions of rectangular matrices.

\bigskip

\bigskip


\vspace{0.6cm}

\textit{Secondary} \textit{49J53, 49R05, 58D19}.

\section{Introduction}
This work revolves around {\em spectral functions}. These are functions on the space of $n\times n$ symmetric matrices $\S^n$ that depend on matrices only through their eigenvalues, that is, functions that are invariant under the action of the orthogonal group by conjugation. Spectral functions can always be written in a composite form $f\circ\lambda$, where $f$ is a permutation-invariant function on $\R^n$ and $\lambda$ is a mapping assigning to each matrix $X$ the vector of eigenvalues $(\lambda_1(X),\ldots,\lambda_n(X))$ in nonincreasing order.

A pervasive theme in the study of such functions is that various variational properties of the permutation-invariant function $f$ are inherited by the induced spectral function $f\circ\lambda$; see e.g. \cite{cov_orig,diff_1,diff_2,high_order,spec_id,spec_prox,man,mather}.
Take convexity for example.
Supposing that $f$ is closed and convex, the main result of \cite{con_herm} shows that the Fenchel conjugate of $f\circ\lambda$ admits the elegant representation
\begin{equation}\label{eqn:fan}
(f\circ\lambda)^{\star}=f^{\star}\circ\lambda.
\end{equation}   
 An immediate conclusion is that $f\circ\lambda$ agrees with its double conjugate and is therefore convex, that is, convexity of $f$ is inherited by the spectral function $f\circ\lambda$. 
An elegant characterization of the subdifferential $\partial (f\circ \lambda)(X)$ in terms of $\partial f(\lambda(X))$ then readily follows \cite[Theorem~3.1]{con_herm} --- an important result for optimization specialists. 

In a follow up paper \cite{eval}, Lewis showed that  
even for nonconvex functions $f$, the following exact relationship holds:
\begin{equation}\label{eqn:maineq_intro}
\partial (f\circ\lambda)(X)=\{U(\Diag v)U^T: v\in\partial f(\lambda(X)),\, U\in \mathcal{O}^n_X\},
\end{equation}
where $$\mathcal{O}^n_X:=\{U\in \mathcal{O}^n: X=U(\Diag \lambda(X))U^T\}.$$
Here, the symbol $\mathcal{O}^n$ denotes the group of orthogonal matrices and the symbols $\partial (f\circ\lambda)$ and $\partial f$ may refer to the Fr\'{e}chet, limiting, or Clarke subdifferentials; see e.g. \cite{RW98} for the relevant definitions. Thus calculating the subdifferential of the spectral function $f\circ\lambda$ on $\S^n$ reduces to computing the subdifferential of the usually much simpler function $f$ on $\R^n$. For instance, subdifferential computation of the $k$'th largest eigenvalue function $X\mapsto \lambda_k(X)$ amounts to analyzing a piecewise polyhedral function, the $k$'th order statistic on $\R^n$ \cite[Section 9]{eval}. Moreover, the subdifferential formula allows one to gauge the underlying geometry of spectral functions, through their ``active manifolds'' \cite{spec_id}, for example. 

In striking contrast to the convex case \cite{con_herm}, the proof of the general subdifferential formula \eqref{eqn:maineq_intro} requires much finer tools,  
and is less immediate to internalize. This paper presents a short, elementary, and revealing derivation of equation \eqref{eqn:maineq_intro} that is no more involved than its convex counterpart. Here's the basic idea. Consider the {\em Moreau envelope} 
$$f_{\alpha}(x):=\inf_y\big\{f(y)+\frac{1}{2\alpha}|x-y|^2\big\}.$$
Similar notation will be used for the envelope of $f\circ\lambda$.
In direct analogy to equation \eqref{eqn:fan}, we will observe that the Moreau envelope satisfies the  equation
$$(f\circ\lambda)_{\alpha}=f_{\alpha}\circ\lambda,$$
and derive a convenient formula for the corresponding  proximal mapping. The case when $f$ is an indicator function was treated in \cite{spec_prox}, and the argument presented here is a straightforward adaptation, depending solely on the Theobald--von Neumann inequality \cite{theo_ineq,von_Neumann}. The key observation now is independent of the eigenvalue setting: membership of a vector $v$ in the proximal or in the Fr\'{e}chet subdifferential of any function $g$ at a point $x$ is completely determined by the local behavior of the univariate function $\alpha \mapsto g_{\alpha}(x+\alpha v)$ near the origin. The proof of the subdifferential formula \eqref{eqn:maineq_intro} quickly flows from there. It is interesting to note that the argument uses very little information about the properties of the eigenvalue map, with the exception of the Theobald--von Neumann inequality. Consequently, it applies equally well in a more general algebraic setting of certain isometric group actions, encompassing also an analogous subdifferential formula for functions of singular values derived in \cite{send,singII,send_thes}; a discussion can be found in the appendix. A different Lie theoretic approach in the convex case appears in \cite{Kon}. 

We complete the paper by reconsidering the second-order theory of spectral functions. In \cite{twice_diff,high_order, diff_2}, the authors derived a formula for the second derivative of a $C^2$-smooth spectral function. In its simplest form it reads
\[\nabla^2F(\Diag a)[B] = \text{Diag} \big ( \nabla^2 f(a) \text{diag}(B)
\big ) + \mathcal{A} \circ B,\]
where $\mathcal{A} \circ B$ is the Hadamard product and 
\[\mathcal{A}_{ij} = \begin{cases} \tfrac{\nabla f(a)_i -\nabla
	f(a)_j}{a_i-a_j} & \text{if $a_i \neq a_j$}\\
\nabla^2f(a)_{ii}-\nabla^2f(a)_{ij} & \text{if $a_i =
	a_j$} \end{cases}.\]
This identity is quite mysterious, and its derivation is quite opaque geometrically. In the current work, we provide a transparent derivation, making clear the role of the invariance properties of the gradient graph. To this end, we borrow some ideas from 
\cite{ diff_2}, while giving them a geometric interpretation.

The outline of the manuscript is as follows. Section~\ref{sec:not} records some basic notation and an important preliminary result about the Moreau envelope (Lemma~\ref{lem:sub_env}). Section~\ref{sec:orth_inv} contains background material on orthogonally invariant functions. 
Section~\ref{sec:der_sub} describes the derivation of the subdifferential formula and Section~\ref{sec:second} focuses on the second-order theory -- the main results of the paper.

\section{Notation}\label{sec:not}
This section briefly records some basic notation, following closely the monograph \cite{RW98}. The symbol $\E$ will always denote an Euclidean space (finite-dimensional real inner product space) with inner product $\langle \cdot,\cdot\rangle$ and induced norm $|\cdot|$. A closed ball of radius $\varepsilon >0$ around a point $x$ will be denoted by $\B_{\varepsilon}(x)$. The closure and the convex hull of a set $Q$ in $\E$ will be denoted by $\cl Q$ and $\conv Q$, respectively.

Throughout, we will consider functions $f$ on $\E$ taking values in the extended real line $\overline{\R}:=\R\cup\{\pm \infty\}$. For such a function $f$ and a point $\bar{x}$, with $f(\bar{x})$ finite, the {\em proximal subdifferential} $\partial_p f(\bar{x})$ consists of all vectors $v\in \E$ such that there exists constants $r>0$ and $\varepsilon >0$ satisfying
$$f(x)\geq f(\bar{x})+\langle v, x-\bar{x}\rangle -\frac{r}{2}|x-\bar{x}|^2\qquad \textrm{ for all } x\in \B_{\varepsilon}(\bar{x}).$$ 
Whenever $f$ is $C^2$-smooth near $\bar{x}$, the proximal subdifferential $\partial_p f(\bar{x})$ consists only of the gradient $\nabla f(\bar{x})$.
A function $f$ is said to be {\em prox-bounded} if it majorizes some quadratic function. In particular, all lower-bounded functions are prox-bounded. For prox-bounded functions, the inequality in the definition of the proximal subdifferential can be taken to hold globally at the cost of increasing  $r$ \cite[Proposition~8.46]{RW98}. The {\em Fr\'{e}chet subdifferential} of $f$ at $\bar{x}$, denoted $\hat{\partial} f(\bar{x})$, consists of all vectors $v\in \E$ satisfying
$$f(x)\geq f(\bar{x})+\langle v, x-\bar{x}\rangle+ o(|x-\bar{x}|).$$ 
Here, as usual, $o(|x-\bar{x}|)$ denotes any term satisfying $\frac{o(|x-\bar{x}|)}{|x-\bar{x}|}\to 0$. Whenever $f$ is $C^1$-smooth near $\bar{x}$, the set $\hat{\partial} f(\bar{x})$
consists only of the gradient $\nabla f(\bar{x})$. The subdifferentials $\partial_p f(\bar{x})$ and $\hat{\partial} f(\bar{x})$ are always convex, while $\hat{\partial} f(\bar{x})$ is also closed. The {\em limiting subdifferential} of $f$ at $\bar{x}$, denoted $\partial f(\bar{x})$, consists of all vectors $v\in \E$ so that there exist sequences $x_i$ and $v_i\in \hat{\partial} f(x_i)$ with $(x_i,f(x_i),v_i)\to (\bar{x},f(\bar{x}),v)$. The same object arises if the vectors $v_i$ are restricted instead to lie in $\partial_p f(x_i)$ for each index $i$; see for example \cite[Corollary 8.47]{RW98}.
The {\em horizon subdifferential}, denoted $\partial^{\infty}f(\bar{x})$, consists of all limits of $\lambda_i v_i$
for some sequences $v_i\in\partial f(x_i)$ and $\lambda_i\geq 0$ satisfying $x_i\to \bar{x}$ and $\lambda_i\searrow 0$. 
This object records horizontal ``normals'' to the epigraph of the function. For example, $f$ is locally Lipschitz continuous around $\bar{x}$ if and only if the set  
$\partial^{\infty}f(\bar{x})$ contains only the zero vector.


The two key constructions at the heart of the paper are defined as follows. Given a function $f\colon\E\to\overline{\R}$ and a parameter $\alpha >0$, the {\em Moreau envelope} $f_{\alpha} $ and the {\em proximal mapping} $P_{\alpha} f$ are defined by 
\begin{align*}
f_{\alpha}(x)&:=\inf_{y\in\E}\, \big\{f(y)+\frac{1}{2\alpha} |y-x|^2\big\},\\
P_{\alpha} f(x) &:= \argmin_{y\in\E} \big\{f(y)+\frac{1}{2\alpha} |y-x|^2\big\}.
\end{align*}
Extending the definition slightly, we will set $f_{0}(x):=f(x)$.
It is easy to see that $f$ is {\em prox-bounded} if and only if there exists some point $x\in \E$ and a real $\alpha >0$ satisfying $f_{\alpha}(x) >-\infty$. 

The proximal and Fr\'{e}chet subdifferentials are conveniently characterized by a differential property of the function $\alpha\mapsto f_{\alpha}(x+\alpha v)$.
 This observation is recorded below. To this end, for any function $\varphi\colon [0,\infty)\to\overline{\R}$, the one-sided derivative will be denoted by
$$\varphi'_+(0):=\lim_{\alpha\searrow 0} \frac{\varphi(\alpha)-\varphi(0)}{\alpha}.$$

\begin{lemma}[Subdifferential and the Moreau envelope]\label{lem:sub_env}{\hfill \\ }
Consider an lsc, prox-bounded function $f\colon\E\to\overline{\R}$, and a point $x$ with $f(x)$ finite.
Fix a vector $v\in \E$ and define the function
$\varphi\colon[0,\infty)\to\overline{\R}$ by setting
 $\varphi(\alpha):=f_{\alpha}(x+\alpha v)$. Then the following are true.
\begin{enumerate}
\item[(i)]\label{it:fre} The vector $v$ lies in $\hat{\partial} f(x)$ if and only if 
\begin{equation}\label{eqn:fre}
\varphi'_+(0)=\frac{|v|^2}{2}.
\end{equation}
\item[(ii)]\label{it:prox} The vector $v$ lies in $\partial_p f(x)$ if and only if there exists $\alpha >0$ satisfying $x\in P_{\alpha}f(x+\alpha v)$, or equivalently
$$\varphi(\alpha)=f(x)+\frac{|v|^2}{2}\alpha.$$
In this case, the equation above continues to hold for all $\tilde{\alpha}\in [0,\alpha]$.
\end{enumerate}
\end{lemma}
\begin{proof}
Claim $(ii)$ is immediate from definitions; see for example \cite[Proposition 8.46]{RW98}. Hence we focus on claim $(i)$. To this end, note first that the inequality
\begin{equation}\label{eqn:basic}
\frac{f_{\alpha}(x+\alpha v)-f(x)}{\alpha}\leq \frac{|v|^2}{2}\qquad \textrm{holds for any }v\in\E.
\end{equation}
Consider now a vector $v\in\hat{\partial} f(x)$ and any sequences $\alpha_i\searrow 0$ and $x_i\in P_{\alpha_i}(x+\alpha_i v)$. 
We may assume $x_i\neq x$ since otherwise there's nothing to prove.
Clearly $x_i$ tend to $x$
and hence 
\begin{align*}
f_{\alpha_i}(x+\alpha_i v)-f(x)&= f(x_i)-f(x)+\frac{1}{2\alpha_i}|(x_i-x)-\alpha_i v|^2 \\
&\geq o(|x_i-x|)+\frac{1}{2\alpha_i}|x_i-x|^2 +\frac{\alpha_i}{2}|v|^2.
\end{align*}
Consequently, we obtain the inequality
$$\frac{f_{\alpha_i}(x+\alpha_i v)-f(x)}{\alpha_i}\geq  \frac{|x_i-x|}{\alpha_i}\cdot\frac{o(|x_i-x|)}{|x_i-x|}+\frac{1}{2}\Big|\frac{x_i-x}{\alpha_i}\Big|^2+\frac{|v|^2}{2}.$$
Taking into account \eqref{eqn:basic} yields the inequality
$$0\geq \frac{|x_i-x|}{\alpha_i}\cdot\left(\frac{o(|x_i-x|)}{|x_i-x|}+\frac{1}{2}\Big|\frac{x_i-x}{\alpha_i}\Big|\right).$$
In particular, we deduce
$\frac{x_i-x}{\alpha_i}\to 0$, and the equation \eqref{eqn:fre} follows.

Conversely suppose that equation \eqref{eqn:fre} holds, and for the sake of contradiction that $v$ does not lie in $\hat{\partial} f(x)$. Then there exists $\kappa >0$ and a sequence $y_i\to x$ satisfying  
$$f(y_i)-f(x)-\langle v,y_i-x\rangle \leq -\kappa |y_i-x|.$$
Then for any $\alpha >0$, observe 
\begin{align*}
 \frac{f_{\alpha}(x+\alpha v)-f(x)}{\alpha}&\leq\frac{1}{\alpha}\Big( f(y_i)-f(x)+\frac{1}{2\alpha}|(y_i-x)-\alpha v|^2\Big)\\
 &\leq -\kappa \frac{|y_i-x|}{\alpha} +\frac{1}{2}\Big|\frac{y_i-x}{\alpha}\Big|^2 +\frac{|v|^2}{2}.
\end{align*}
 Setting $\alpha_i:=\frac{|y_i-x|}{\kappa}$ and letting $i$ tend to $\infty$ yields a contradiction.
\end{proof}

%


\section{Symmetry and orthogonal invariance}\label{sec:orth_inv}
Next we recall a basic correspondence between symmetric functions and spectral functions of symmetric matrices. The discussion follows that of \cite{eval}. Henceforth $\R^n$ will denote an $n$-dimensional real Euclidean space with a specified basis. Hence one can associate $\R^n$ with a collection of $n$-tuples $(x_1,\ldots,x_n)$, in which case the inner product $\langle \cdot,\cdot \rangle$ is the usual dot product. The finite group of coordinate permutations of $\R^n$ will be denoted by $\Pi^n$. A function $f\colon\R^n\to\overline{\R}$ is {\em symmetric} whenever it is $\Pi^n$-invariant, meaning 
$$f(\pi x) =f(x) \quad \textrm{ for all }x\in \R^n \textrm{ and } \pi\in \Pi^n.$$ 
It is immediate to verify that if $f$ is symmetric, then so is the Moreau envelope $f_{\alpha}$ for any $\alpha\geq 0$. This elementary observation will be important later.

The vector space of real $n\times n$ symmetric matrices will be denoted by $\S^n$ and will be endowed with the trace inner product 
$\langle X,Y\rangle=\tr XY$, and the induced Frobenius norm $|X|=\sqrt{\tr X^2}$. For any $x\in\R^n$, the symbol $\Diag x$ will denote the $n\times n$ matrix with $x$ on its diagonal and with zeros off the diagonal, while for a matrix $X\in {\bf S}^n$, the symbol $\diag X$ will denote the $n$-vector of its diagonal entries.

The group of real $n\times n$ orthogonal matrices will be written as $\mathcal{O}^n$.
The eigenvalue mapping $\lambda\colon\S^n\to \R^n$ assigns to each matrix $X$ in $\S^n$ the vector of its eigenvalues $(\lambda_1(X),\ldots, \lambda_n(X))$ in a nonincreasing order. A function $F\colon \S^n\to\overline{\R}$ is {\em spectral} if it is $\mathcal{O}^n$-invariant under the conjugation action, meaning 
$$F(UXU^T) =F(X) \quad \textrm{ for all }X\in \S^n \textrm{ and } U\in \mathcal{O}^n.$$
In other words, spectral functions are those that depend on matrices only through their eigenvalues. A basic fact is that any spectral function $F$ on $\S^n$ can be written as a composition of $F=f\circ\lambda$ for some symmetric function $f$ on $\R^n$. Indeed, $f$ can be realized as the restriction of $F$ to diagonal matrices $f(x)=F(\Diag x)$.

Two matrices $X$ and $Y$ in $\S^n$ are said to admit a {\em simultaneous spectral decomposition} if
there exists an orthogonal matrix $U\in \mathcal{O}^n$ such that 
$UXU^T$ and $UYU^T$ are both diagonal matrices. It is well-known that this condition holds if and only if $X$ and $Y$ commute. The matrices $X$ and $Y$ are said to admit a {\em simultaneous ordered spectral decomposition} if
there exists an orthogonal matrix $U\in \mathcal{O}^n$ satisfying
$UXU^T=\Diag\lambda(X)$ and $UYU^T=\Diag\lambda(Y)$.
The following result characterizing this property, essentially due to Theobald \cite{theo_ineq} and von Neumann \cite{von_Neumann}, plays a central role in spectral variation analysis.

\begin{theorem}[Von Neumann-Theobald]\label{thm:fan}
	Any two matrices $X$ and $Y$ in ${\bf S}^n$ satisfy the inequality
	$$|\lambda(X)-\lambda(Y)|\leq |X-Y|.$$
	Equality holds if and only if $X$ and $Y$ admit a simultaneous ordered spectral decomposition.
\end{theorem}

This result is often called a trace inequality, since the eigenvalue mapping being 1-Lipschitz (as in the statement above) is equivalent to the inequality $$\langle \lambda(X),\lambda(Y) \rangle\geq \langle X,Y \rangle  \qquad \textrm{ for all } X,Y\in{\bf S}^n.$$

\section{Derivation of the subdifferential formula}\label{sec:der_sub}
In this section, we derive the subdifferential formula for spectral functions. In what follows, for any matrix $X\in \S^n$ define the diagonalizing matrix set
$$\mathcal{O}_X:=\{U\in\mathcal{O}^n: U(\Diag\lambda(X))U^T=X\}.$$
The spectral subdifferential formula readily follows from Lemma~\ref{lem:sub_env} and the following intuitive proposition, a proof of which can essentially be seen in \cite[Proposition 8]{spec_prox}. 

\begin{theorem}[Proximal analysis of spectral functions]\label{thm:prox_spec} \hfill \\
Consider a symmetric function $f\colon\R^n\to\overline {\R}$. Then the equation 
\begin{equation}
	(f\circ \lambda)_{\alpha}=f_{\alpha}\circ\lambda \qquad \textrm{ holds}. \label{1_prox}
\end{equation}	
In addition, the proximal mapping admits the representation:
\begin{equation}
	P_{\alpha}(f\circ \lambda)(X)=\big\{U\big(\Diag\, y\big)U^T: y\in P_{\alpha}f(\lambda(X)),\, U\in \mathcal{O}_X\big\}. \label{2_prox}
\end{equation}
	Moreover, for any $Y\in P_{\alpha}(f\circ \lambda)(X)$ the matrices $X$ and $Y$ admit a simultaneous ordered spectral decomposition.
\end{theorem}
\begin{proof}
	For any $X$ and $Y$, applying the trace inequality (Theorem~\ref{thm:fan}), we deduce
	\begin{align}
 f(\lambda(Y))+\frac{1}{2\alpha} |Y-X|^2\geq
 f(\lambda(Y))+\frac{1}{2\alpha}|\lambda(Y)-\lambda(X)|^2\geq f_{\alpha}(\lambda(X)).\label{eqn:main_eq}
	\end{align}
	Taking the infimum over $Y$, we deduce $(f\circ \lambda)_{\alpha}(X)\geq f_{\alpha}(\lambda(X))$.
	On the other hand, for any $U\in \mathcal{O}_X$, the inequalities hold:
	\begin{align*}
	(f\circ \lambda)_{\alpha}(X)&=	\inf_{Y}\, \big\{f(\lambda(Y))+\frac{1}{2\alpha} |Y-X|^2\big\}\\
	&= \inf_{Y}\, \big\{f(\lambda(Y))+\frac{1}{2\alpha} |U^TYU-\Diag\lambda(X)|^2\big\}\leq f_{\alpha}(\lambda(X)).
	\end{align*}
	This establishes \eqref{1_prox}. 
	
	To establish equation \eqref{2_prox}, consider first a matrix $U\in\mathcal{O}_X$ and a vector $y\in P_{\alpha} f(\lambda(X))$, and define $Y:=U(\Diag y)U^T$. Then we have
	\begin{align*}
	(f\circ\lambda)(Y)+&\frac{1}{2\alpha}|Y-X|^2=f(y)+\frac{1}{2\alpha}|y-\lambda(X)|^2
	=f_{\alpha}(\lambda(X))=(f\circ\lambda)_{\alpha}(X).
	\end{align*}
	Hence the inclusion $Y\in P_{\alpha}(f\circ \lambda)(X)$ is valid, as claimed.	
	Conversely, fix any matrix $Y\in P_{\alpha}(f\circ \lambda)(X)$. Then plugging in $Y$  into \eqref{eqn:main_eq}, the left-hand-side equals $(f\circ \lambda)_{\alpha}(X)$ and hence the two inequalities in \eqref{eqn:main_eq} hold as equalities. 
The second equality immediately yields the inclusion $\lambda(Y)\in P_{\alpha} f(\lambda(X))$, while the first along with Theorem~\ref{thm:fan} implies 
 that $X$ and $Y$ admit a simultaneous ordered spectral decomposition, as claimed.
\end{proof}

Combining Lemma~\ref{lem:sub_env} and Theorem~\ref{thm:prox_spec}, the main result of the paper readily follows.	
\begin{theorem}[Subdifferentials of spectral functions]\label{thm:submain}
Consider an lsc symmetric function $f\colon\R^n\to\overline {\R}$. Then the following equation holds:
\begin{equation}\label{eqn:main}
	\partial (f\circ \lambda)(X)=\big\{U\big(\Diag v\big)U^T: v\in\partial f(\lambda(X)), \, U\in \mathcal{O}_{X} \big\}.
\end{equation}
Analogous formulas hold for the proximal, Fr\'{e}chet, and horizon subdifferentials.
\end{theorem}
\begin{proof}
Fix a matrix $X$ in the domain of $f\circ\lambda$ and define $x:=\lambda(X)$. Without loss of generality, suppose that $f$ is lower-bounded.	
Indeed if this were not the case, then since $f$ is lsc there exists $\varepsilon >0$ so that $f$ is lower-bounded on the ball $\mathcal{B}_{\epsilon}(x)$.
Consequently adding to $f$ the indicator function of the symmetric set $\cup_{\pi\in \Pi} \mathcal{B}_{\epsilon}(\pi x)$ assures that the function is lower-bounded.

We first dispense with the easy inclusion $\subseteq$ for all the subdifferentials. To this end, recall that if $V$ is a proximal subgradient of $f\circ \lambda$ at $X$, then there exists $\alpha >0$ satisfying $X\in P_{\alpha} (f\circ \lambda)(X+\alpha V)$. Theorem~\ref{thm:prox_spec} then implies that $X$ and $V$ commute. Taking limits, we deduce that all Fr\'{e}chet, limiting, and horizon subgradients of $f\circ\lambda$ at $X$ also commute with $X$.
  Recalling that commuting matrices admit simultaneous spectral decomposition, basic definitions immediately yield the inclusion $\subseteq$ in equation \eqref{eqn:main} for the proximal and for the Fr\'{e}chet subdifferentials. 
  Taking limits, we deduce the inclusion $\subseteq$ in \eqref{eqn:main} for the limiting and for the horizon subdifferentials, as well. 


Next, we argue the reverse inclusion. To this end, define 
$V:=U(\Diag v)U^T$ for an arbitrary matrix  $U\in \mathcal{O}_X$ and any vector $v\in\R^n$. 
Then Theorem~\ref{thm:prox_spec}, along with the symmetry of the envelope $f_{\alpha}$, yields the equation
$$\frac{(f\circ\lambda)_{\alpha}(X+\alpha V)-f(\lambda(X))}{\alpha}=\frac{f_{\alpha}(x+\alpha v)-f(x)}{\alpha}.$$
Consequently if $v$ lies in $\partial_p f(x)$, then Lemma~\ref{lem:sub_env} shows that for some $\alpha >0$ the right-hand-side equals $\frac{|v|^2}{2}$, or equivalently $\frac{|V|^2}{2}$. Lemma~\ref{lem:sub_env} then yields the inclusion $V\in \partial_p (f\circ\lambda)(X)$. Similarly if $v$ lies in $\hat{\partial} f(x)$, then the same argument but with $\alpha$ tending to $0$ shows that $V$ lies in $\hat{\partial} (f\circ\lambda)(X)$.
 Thus the inclusion $\supseteq$ in equation \eqref{eqn:main} holds for the proximal and for the Fr\'{e}chet subdifferentials. Taking limits, the same inclusion holds for the limiting and for the horizon subdifferentials. This completes the proof.
\end{proof}

\begin{remark}
	It easily follows from Theorem~\ref{thm:submain} that the inclusion $\supseteq$ holds for the Clarke subdifferential. The reverse inclusion, however, requires a separate argument given in \cite[Sections 7-8]{eval}.
\end{remark}

In conclusion, we should mention that all the arguments in the section apply equally well for Hermitian matrices (with the standard Hermitian trace product), with the orthogonal matrices replaced by unitary matrices. Entirely analogous arguments also apply for functions of singular values of rectangular matrices (real or complex). For more details, see the appendix in the arXiv version of the paper.

\section{Hessians of $C^2$-smooth spectral functions}\label{sec:second}
In this section, we revisit the second-order theory of spectral functions. To this end, fix for the entire section an lsc symmetric function $f\colon\R^n\to\overline{\R}$ and define the spectral function $F:=f\circ\lambda$ on $\mathbf{S}^n$.
It is well known that $f$ is $C^2$-smooth around a matrix $X$ if and only if $F$ is $C^2$-smooth around $\lambda(X)$; see \cite{diff_1,diff_2,high_order,twice_diff}. Moreover, a formula for the Hessian of $F$ is available:
for matrices $A = \text{Diag}(a)$ and $B \in \mathbf{S}^n$ we have 
\[\nabla^2F(A)[B] = \text{Diag} \big ( \nabla^2 f(a) \text{diag}(B)
\big ) + \mathcal{A} \circ B,\]
where $\mathcal{A} \circ B$ is the Hadamard product and 
\[\mathcal{A}_{ij} = \begin{cases} \tfrac{\nabla f(a)_i -\nabla
	f(a)_j}{a_i-a_j} & \text{if $a_i \neq a_j$}\\
\nabla^2f(a)_{ii}-\nabla^2f(a)_{ij} & \text{if $a_i =
	a_j$} \end{cases}.\]
The assumption that $A$ is a diagonal matrix is made without loss of generality, as will be apparent shortly. 
In this section, we provide a transparent geometric derivation of the Hessian formula by considering invariance properties of $\gph \nabla F$. Some of our arguments give a geometric interpretation of the techniques in \cite{diff_2}.

\begin{remark}[Hessian and the gradient graph]\label{rem:hess_grad}
{\rm
Throughout the section we will appeal to the following basic property of the Hessian. For any $C^2$-smooth function $g$ on an Euclidean space, the vector $z:=\nabla^2 g(a)[b]$ is the unique vector satisfying $(z,-b)\in N_{\gph \nabla g} (a,\nabla g(a))$.}
\end{remark}

Consider now the action of the orthogonal group $\mathcal{O}^n$
on $\mathbf{S}^n$ by conjugation namely $U.X=UXU^T$. Recall that $F$ is {\em invariant} under this action, meaning $F(U.X)=F(X)$ for all orthogonal matrices $U$. This action naturally extends to the product space $\mathbf{S}^n\times\mathbf{S}^n$ by setting 
$U.(X,Y)=(U.X,U.Y)$. As we have seen, the graph $\gph\nabla F$ is then {\em invariant} with respect to this action:
$$U.\gph \nabla F=\gph \nabla F \qquad \textrm{ for all } U\in \mathcal{O}^n.$$
One immediate observation is that $N_{\gph \nabla F}(U.X,U.Y)=U.N_{\gph \nabla F}(X,Y)$. Consequently we deduce
$$(Z,-B)\in N_{\gph \nabla F}(X,Y)\quad\Longleftrightarrow\quad   (U.Z,-U.B)\in N_{\gph \nabla F}(U.X,U.Y)$$
The formula 
\begin{equation}\label{eqn:equivariance}
\nabla^2 F(X)[B]=U^T.\nabla^2 F(U.X)[U.B]
\end{equation}
 now follows directly from Remark~\ref{rem:hess_grad}, whenever $F$ is $C^2$-smooth around $X$. As a result, when speaking about the operator $\nabla^2 F(X)$, we may assume without loss of generality that $X$ and $\nabla F(X)$ are both diagonal matrices.


Next we briefly recall a few rudimentary properties of the conjugation action; see for example \cite[Sections 4, 8, 9]{Lee2}. We say that a $n \times n$ matrix $W$ is
 \textit{skew-symmetric} if $W^T=-W$. Then it is well-known that $\mathcal{O}^n$ is a smooth manifold and the tangent space to $\mathcal{O}^n$ at the identity matrix consists of skew-symmetric matrices:
 $$T_{\mathcal{O}^n}(I)=\{W\in \R^{n\times n}: W \textrm{ is skew-symmetric}\}.$$
The \textit{commutator} of two matrices $A,B\in\R^{n\times n}$, denoted by $[A,B]$ is the matrix
$[A,B] := AB - BA$. An easy computation shows that the commutator of
a skew-symmetric matrix with a symmetric matrix is itself symmetric. Moreover, 
the identity $$\langle X,[W,Z] \rangle=\langle [X,W],Z \rangle$$ holds for any matrices $X,Z\in {\bf S}^n$ and skew-symmetric $W$.
For any matrix $A \in \mathbf{S}^n$, the \textit{orbit
	of $A$}, denoted by $\mathcal{O}^n . A$ is the set
\[\mathcal{O}^n . A = \{U . A \, : \, U \in \mathcal{O}^n\}.\]
Similarly, the orbit of a pair $(A,B)\in\mathbf{S}^n\times \mathbf{S}^n$ is the set
\[\mathcal{O}^n . (A,B) = \{(U . A, U. B) \, : \, U \in \mathcal{O}^n\}.\]
An standard computation\footnote{Compute the differential of the mapping $\mathcal{O}^n \ni U\mapsto U.A$} now shows that orbits are smooth manifolds with tangent spaces
\begin{align*}
T_{\mathcal{O}^n . A}(A)&=\{[W,A]: W \textrm{ is skew-symmetric}\},\\
T_{\mathcal{O}^n . (A,B)}(A,B)&=\{([W,A],[W,B]): W \textrm{ is skew-symmetric}\}.
\end{align*}


Now supposing that $F$ is twice differentiable at a matrix $A\in {\bf S}^{n\times n}$, the graph $\gph \nabla F$ certainly contains the orbit $\mathcal{O}^n.(A,\nabla F(A))$. In particular, this implies that the tangent space to $\gph \nabla F$ at $(A,\nabla F(A))$ contains the tangent space to the orbit:
$$\{ ([W,A], [W, \nabla F(A)] ) :
\, W \text{ skew-symmetric}\}.$$
Thus for any $B \in
\mathbf{S}^n$, the tuple $(\nabla^2F(A)[B], -B)$ is
orthogonal to the tuple $([W,A], [W, \nabla F(A)] )$ for any skew-symmetric matrix $W$. We record this elementary observation in the following lemma. This also appears as \cite[Lemma 3.2]{diff_2}.

\begin{lemma}[Orthogonality to orbits] \label{lem:skew} Suppose $F$ is $C^2$-smooth around $A\in{\bf S}^n$. Then for any skew-symmetric matrix $W$ and any $B \in
	\mathbf{S}^n$, we have
	\[\ip{\nabla^2 F(A)[B], [W,A]} = \ip{B, [W, \nabla
		F (A)]}.\]
\end{lemma}

\begin{proof} This is immediate from the preceding discussion. 
\end{proof}

Next recall that the \textit{stabilizer}  of a matrix $A\in{\bf S}^n$ is the set:
\[\text{Stab}(A) = \{U \in \mathcal{O}^n \, : \, U.A = A\}.\]
Similarly we may define the set \text{Stab}(A,B).

\begin{lemma}[Tangent space to the stabilizer]\label{lem:tan_stab}
For any matrices $A,B\in {\bf S}^n$, the tangent spaces to $\text{Stab}(A)$ and to $\text{Stab}(A, B)$  at the identity matrix are the sets
$$\{W\in \R^{n\times n} \, : \, W \text{ skew-symmetric, } [W,A] = 0\},$$
$$\big \{  W\in \R^{n\times n} \, \, : \, W  \text{ skew-symmetric}, [W,A] = [W,B]= 0 \big
	\},$$
respectively.
\end{lemma}
\begin{proof}[(Proof sketch)]
Define the orbit map $\theta^{(A)}\colon \mathcal{O}^n\to \mathcal{O}^n.A$ by 
setting $\theta^{(A)}(U):=U.A$. A quick computation shows that 
$\theta^{(A)}$ is equivariant with respect to left-multiplication action of $\mathcal{O}^n$ on itself and the conjugation action of $\mathcal{O}^n$ on $\mathcal{O}^n.A$.
Hence the equivariant rank theorem (\cite[Theorem 7.25]{Lee2}) implies that
$\theta^{(A)}$ has constant rank. In fact, since $\theta^{(A)}$ is surjective, it is a submersion. It follows that the stabilizer 
$$\text{Stab}(A)=(\theta^{(A)})^{-1}(A)$$
is a smooth manifold with tangent space at the identity equal to the kernel of the differential $d \,
\theta^{(A)}\big|_{U=I}(W)=[W,A]$. 
The expression for the tangent space to $\text{Stab}(A)$ immediately follows. The analogous expression for $\text{Stab}(A,B)$ follows along similar lines.
\end{proof}

With this, we are able to state and prove the main theorem.
\begin{theorem}[Hessian of $C^2$-smooth spectral functions]\label{thm:main_sec_ord} Consider a symmetric function $f\colon\R^n\to\R$ and the spectral function $F= f\circ \lambda$. Suppose  that $F$ is $C^2$-smooth around a matrix $A := \text{Diag}(a)$ and for any matrix matrix $B \in {\bf S}^n$ define $Z := \nabla^2F(A)[B]$. Then equality $$\text{diag}(Z) = \nabla^2 f(a)[\text{diag}(B)],$$ holds, while for 
		indices $i\neq j$, we have
		
\[Z_{ij}  =		\begin{cases} B_{ij} \left ( \frac{\nabla f(a)_i-\nabla
			f(a)_j}{a_i-a_j}\right ) &\mbox{if } a_i \neq a_j \\ 
 B_{ij} \big (\nabla^2f(a)_{ii}-\nabla^2f(a)_{ij}\big ) & \mbox{if } a_i = a_j. \end{cases} \]
\end{theorem}

\begin{proof} 
	First observe that clearly $f$ must be $C^2$ smooth at $a$.	
	Now, since $A$ is diagonal, so is the gradient $\nabla F(A)$. So without loss of generality, we can assume $\nabla F(A) = \text{Diag}(\nabla f(a))$. 
	
	Observe now that $(Z, -B)$ is orthogonal to the tangent space of $\gph \nabla F$ at $(A, \nabla F(A))$. On the other hand, for any vector $a'\in \R^n$, we have equality
	\begin{align*}\ip{ \begin{pmatrix} Z\\ -B 
			\end{pmatrix}, \begin{pmatrix} \text{Diag}(a')-\text{Diag}(a)
			\\  \text{Diag}(\nabla f(a') )-\text{Diag}( \nabla f(a)
			) \end{pmatrix}}=
		\ip{\begin{pmatrix} \text{diag}(Z) \\
				-\text{diag}(B) \end{pmatrix}, \begin{pmatrix} a'-a \\ \nabla
				f(a') - \nabla
				f(a) \end{pmatrix} }.
\end{align*}
It follows immediately that the tuple $(\text{diag}(Z), -\text{diag}(B))$ is orthogonal to the tangent
space of $\gph \nabla f$ at $(a, \nabla f(a))$. Hence we deduce the equality
$\text{diag}(Z)=\nabla^2 f(a)[\text{diag}(B)]$ as claimed.

Next fix  indices $i$ and $j$ with $a_i\neq a_j$, and define the skew-symmetric matrix $W^{(i,j)}:=e_ie_j^T-e_je_i^T$, where $e_k$ denotes the $k$'th standard basis vector.  
Applying Lemma~\ref{lem:skew} with the skew-symmetric matrix $W=\tfrac{1}{a_i-a_j}  W^{(i,j)}$, we obtain 

\begin{align*}
-2Z_{ij} &= \ip{Z, \big [\tfrac{1}{a_i-a_j}  W^{(i,j)}, A \big ]} =
-\ip{ [\tfrac{1}{a_i-a_j}  W^{i,j}, B \big],
	\nabla F(A)} \\
	&= -\ip{ \textrm{diag}[\tfrac{1}{a_i-a_j}  W^{i,j}, B \big],
	\nabla f(a)}=-2B_{ij}\left (\frac{\nabla f(a)_i-\nabla f(a)_j}{a_i-a_j} \right ).
\end{align*}
The claimed formula $Z_{ij}=B_{ij}\left (\frac{\nabla f(a)_i-\nabla f(a)_j}{a_i-a_j} \right )$ follows.

Finally, fix indices $i$ and $j$, with $a_i=a_j$. Observe now the inclusion $$\text{Stab}(A) \subset \text{Stab}(\nabla
F(A)).$$ Indeed for any matrix $U \in \text{Stab}(A)$,  we have
\[\nabla F(A) = \nabla F(UAU^T) = U \nabla F(A) U^T.\]
This in particular immediately implies that the tangent space $T_{\gph \nabla F}(A,\nabla F(A))$ is invariant under the action of $\text{Stab}(A)$, that is 
 $$U.T_{\gph \nabla F}(A,\nabla F(A))=T_{\gph \nabla F}(A,\nabla F(A))$$
for any $U\in \text{Stab}(A)$. Hence their entire orbit $\text{Stab}(A).(X,Y)$ of any tangent vector
$(X,Y)\in T_{\gph \nabla F}(A,\nabla F(A))$ is contained in the tangent space 
$T_{\gph \nabla F}(A,\nabla F(A))$. We conclude that the tangent space to such an orbit $\text{Stab}(A).(X,Y)$ at $(X,Y)$ is contained in $T_{\gph \nabla F}(A,\nabla F(A))$ as well.


Define now the matrices $E_i := \text{Diag}(e_i)$ and $\hat{Z} := \text{Diag}(\nabla^2
f(a)[e_i])$. Because $F$ is $C^2$-smooth, clearly the inclusion $(E_i, \hat{Z}) \in T_{\text{gph}
\nabla F} (A, \nabla F(A)$ holds. The above argument, along with Lemma~\ref{lem:tan_stab}, immediately implies the inclusion 
\[\{([W,E_i], [W, \hat{Z}] ) : \, W
\text{ skew-symmetric}, [W, A] =0\} \quad\subseteq\quad T_{\text{gph} \nabla F}(A, \nabla
	F(A))\]
and in particular, $([W, E_i], [W, \hat{Z}])$ is orthogonal to
$(Z, -B)$ for any skew-symmetric $W$  satisfying $[W,A]=0$. To finish the proof, simply set $W = W^{(i,j)}$. Then since $a_i=a_j$, we have $[W, A] = 0$ and therefore 
\begin{align*}
-2Z_{ij} = \ip{Z, [W^{(i,j)}, E_i]} &= \langle B, [W^{(i,j)},\hat{Z}]\rangle =-\ip{[W^{(i,j)}, B], \hat{Z}} \\
&=-2
B_{ij} \big (\nabla^2f(a)_{ii}-\nabla^2f(a)_{ij} \big ),
\end{align*}
as claimed. This completes the proof.
\end{proof}

\begin{remark}
The appealing geometric techniques presented in this section seem promising for obtaining at least necessary conditions for the generalized Hessian, in the sense of \cite{Mord_1}, of spectral functions that are not necessarily $C^2$-smooth. Indeed the arguments presented deal entirely with the graph $\gph \nabla f$, a setting perfectly adapted to generalized Hessian computations. There are difficulties, however. To illustrate, consider a matrix $Z\in \partial^2 F(A|V)$. Then one can easily establish properties of $\Diag Z$ analogous to those presented in Theorem~\ref{thm:main_sec_ord}, as well as properties of $Z_{ij}$ for indices $i$ and $j$ satisfying $a_i\neq a_j$. The difficulty occurs for indices $i$ and $j$ with $a_i=a_j$. In this case, our argument used explicitly the fact that tangent cones to $\gph \partial f$ are linear subspaces, a property that is decisively false in the general setting. 
\end{remark}

\appendix\section{Comments on isometric group actions}\label{sec:abst}
It is clear from the Section 1-4, that there is a richer underlying structure governing the results of Theorems~\ref{thm:prox_spec} and \ref{thm:submain}, with the trace inequality (Theorem~\ref{thm:fan}) playing an essential role. This appendix outlines a rudimentary algebraic framework in which the previous arguments can be understood, unifying the eigenvalue and the singular value pictures \cite{send}, while leaving room for new settings to be explored.

Fix a metric space $\mathcal{V}$ and a group $\mathcal{G}$ acting on $\mathcal{V}$ by isometries. 
Let $\mathcal{H}$ be another metric space injecting isometrically by a mapping $i\colon\mathcal{H}\hookrightarrow\mathcal{V}$ into $\mathcal{V}$. Intuitively, $\mathcal{H}$
is a subset of $\mathcal{V}$ with $i$ the canonical injection. Notationally, however, it is cleaner to consider $\mathcal{H}$ as a separate entity.
Without loss of generality, we will use the symbol $d(\cdot,\cdot)$ to denote the metric both in $\mathcal{V}$ and in $\mathcal{H}$. Fix also a distinguished $\mathcal{G}$-invariant mapping $p\colon \mathcal{V}\to\mathcal{H}$. 
The diagram summarizes the notation:
$$\mathcal{H} \mathrel{\mathop{\rightleftarrows}^{i}_{p}} \mathcal{V}$$
It is instructive to keep in mind the following motivating examples: 
$$\R^n \mathrel{\mathop{\rightleftarrows}^{\Diag}_{\lambda}} {\bf S}^n,\qquad \R^n \mathrel{\mathop{\rightleftarrows}^{\Diag}_{\lambda}} {\bf H}^n, \qquad \R^m \mathrel{\mathop{\rightleftarrows}^{\Diag}_{\sigma}} {\bf R}^{m\times n},\qquad \R^m \mathrel{\mathop{\rightleftarrows}^{\Diag}_{\sigma}} {\bf C}^{m\times n}.$$
In the first example (the focus of the previous sections), the group $\mathcal{G}=\mathcal{O}^n$ acts by conjugation $U.X=UXU^T$. In the second example,  
${\bf H}^n$ is the space of $n\times n$ Hermitian matrices (with the standard Hermitian inner production $\langle X,Y\rangle=\textrm{re } X^*Y$), and $\mathcal{G}$ is the unitary group acting by $U.X= U X U^*$. In the third example $\R^{m\times n}$ is the space of real $m\times n$ matrices (with the trace product $\langle X,Y\rangle=\tr X^TY$), the group $\mathcal{G}=\mathcal{O}^m\times\mathcal{O}^n$ acts by 
$(U,V).X=UXV^T$, and $\sigma$ is the mapping assigning to each $m\times n$ matrix its vector of singular values in a nonincreasing order. The fourth example is analogous.
 The goal of this section is to isolate the shared features of the four examples above that make a subdifferential formula along the lines of \eqref{eqn:main} possible. That is, we aim to investigate conditions on $p$ under which one can effectively treat $\mathcal{G}$-invariant functions $F\colon\mathcal{V}\to\overline{\R}$ by instead considering their restrictions $F\circ i\colon\mathcal{H}\to\overline{\R}$. 
Some notational abstraction will greatly help simplify the ensuing formulas. To this end, following standard terminology, the {\em pullback} of any mapping $F$ on $\mathcal{V}$ is the mapping $F^*:=F\circ i$ defined now on $\mathcal{H}$. Similarly the pullback of a mapping $f$ on $\mathcal{H}$ is the mapping $f^*:=f\circ p$ on $\mathcal{V}$. The {\em pushforward} of $p$ is the mapping $p_*=i\circ p\colon\mathcal{V}\to\mathcal{V}$.
For instance in the first example, for any function $F\colon\S^n\to\overline{\R}$, the pullback $F^*(x)$ is the diagonal restriction $x\mapsto F(\Diag(x))$; the pullback of a function $f$ on $\R^n$ is the spectral mapping $f^*=f\circ\lambda$; and the pullback of $\lambda$ is the reordering mapping $\uparrow\colon\R^n\to\R^n$, meaning
$x^{\uparrow}$ is obtained by permuting coordinates of  $x$ to be nonincreasing.
The following definition identifies the salient properties needed, in light of the current paper, for effective treatment of $\mathcal{G}$-invariant functions $F\colon\mathcal{V}\to\overline{\R}$ by means of their restrictions $F^*\colon\mathcal{H}\to\overline{\R}$. For clarity, elements of $\mathcal{H}$ will be denoted with lower-case letters, while elements of $\mathcal{V}$ will be denoted with upper-case letters. 
\begin{definition}[Metric reduction]
The space $\mathcal{V}$ {\em metrically reduces} to $\mathcal{H}$ if the following compatibility conditions hold: 
\begin{enumerate}
\item {\bf (Idempotence)} $p^*\circ p^*=p^*$;
\item {\bf (Orbit preservation)} $p_*(X)$ lies in the $\mathcal{G}$-orbit of $X$, for all $X\in \mathcal{V}$;
\item {\bf (Non-expansiveness)} $d\big(p(X),p(Y))\leq d\big(X,Y)\qquad \textrm{ for all } X,Y\in\mathcal{V}$;
\end{enumerate}
The reduction is {\em faithful} if in addition the following is true for all $X,Y\in \mathcal{V}$:
$$d\big(p(X),p(Y))= d\big(X,Y)\quad \Longrightarrow\quad \exists g\in\mathcal{G} \textrm{ with } gX=p_*(X) \textrm{ and }gY=p_*(Y).$$
\end{definition}
\smallskip
An appropriate notion of symmetry on $\mathcal{H}$ that is compatible with $\mathcal{G}$-invariance on $\mathcal{V}$ is as follows. A function $f\colon\mathcal{H}\to\overline{\R}$ is $p${\em -symmetric} whenever 
$$f(p^*(x))=f(x) \quad\textrm{ for any } x\in \mathcal{H}.$$
In the spectral example, $\S^n$ faithfully reduces to $\R^n$ as a consequence of Theorem~\ref{thm:fan};  $\mathcal{G}$-invariant functions are what we called spectral, while $\lambda$-symmetric functions are what we called symmetric. The other three running examples are analogous. 
\begin{lemma}[Invariance and symmetry]
The following two properties of a function $F\colon\mathcal{V}\to\overline{\R}$ are equivalent.
\begin{enumerate}
\item[(i)] $F$ is $\mathcal{G}$-invariant.
\item[(ii)] $F=f\circ p$ for some $p$-symmetric function $f$ on $\mathcal{H}$.
\item[(iii)] $F=(F^*)^*$
\end{enumerate}
\end{lemma}
\begin{proof}
Suppose that $(i)$ holds and define $f:=F^*$. Then observe 
$f^*(X)=F\circ p_*(X)$. By the orbit preservation property, there exists some $g\in \mathcal{G}$ satisfying $gx=p_*(X)$ and hence $f^*(X)=F(X)$ for all $X\in \mathcal{V}$. Hence implication $(iii)$ holds. Suppose now that $(iii)$ holds, meaning
$F(X)=F^*(p(X))$ for all $X\in\mathcal{V}$. Hence, in particular 
$F^*(y)=F(i(y))=(F^*)^*(i(y))=F^*(p\circ i(y))=F^*(p^*(y))$ for all $y\in \mathcal{H}$. By definition then $F^*$ is $p$-symmetric and $(ii)$ follows. The final implication $(ii)\Rightarrow (i)$ is trivial since $p$ is $\mathcal{G}$-invariant.
 \end{proof}
For notational convenience, henceforth, for any point $y\in\mathcal{H}$ the corresponding capital letter $Y$ will stand for $i(y)$.
Observe that the Moreau envelopes and proximal mappings of functions on $\mathcal{V}$ and on $\mathcal{H}$ have obvious meanings. A proof nearly identical to that of Theorem~\ref{thm:prox_spec} shows that if $\mathcal{V}$ metrically reduces to $\mathcal{H}$, then for any lsc $p$-symmetric function $f\colon\mathcal{H}\to\overline{\R}$ the commutatively relation holds:
$$(f^*)_{\alpha}=(f_{\alpha})^*.$$
Assuming in addition that the reduction is faithful, the equation holds:
\begin{equation*}
	P_{\alpha}f^*(X)=\big\{g^{-1}Y: y\in P_{\alpha}f(p(X)),\, g\in \mathcal{G}_X\big\},
\end{equation*}
where $$\mathcal{G}_{X}:=\{g\in\mathcal{G}: p_*(X)=gX\}.$$
Moreover, for any $Z\in P_{\alpha}f^*(X)$ there exists $g\in \mathcal{G}$ satisfying  $p_*(Z)=gZ$ and  $p_*(X)=gX$. 
Suppose moreover that $\mathcal{V}$ and $\mathcal{H}$ are Euclidean spaces with $i$ a linear mapping, and that $\mathcal{G}$ is a compact subgroup of linear isometries. Then a proof identical to that of Theorem~\ref{thm:submain} shows that the following formula holds:
\begin{equation*}
\partial f^*(X)=\big\{g^{-1}V: v\in \partial f(p(X)),\, g\in \mathcal{G}_X\big\},
\end{equation*}
The four running examples of the section fit nicely into this framework.

\bibliography{bibliography}{}
\bibliographystyle{plain}

\end{document}